\documentclass[journal]{IEEEtran}
\usepackage{amssymb,amsmath,amsthm,amsfonts}
\usepackage{color}
\usepackage{epstopdf}
\usepackage{graphicx}
\usepackage{subfigure}
\def\crr{\cr\noalign{\vskip2mm}}
\def\dref#1{(\ref{#1})}
\def\disp{\displaystyle}
\newtheorem{remark}{Remark}[section]
\newtheorem{theorem}{Theorem}[section]                   

\newtheorem{lemma}{Lemma}[section]

\ifCLASSINFOpdf
\else
\fi
\begin{document}

\title{On Convergence of Tracking Differentiator with  Multiple Stochastic Disturbances}

\author{Ze-Hao Wu,  Hua-Cheng Zhou, Bao-Zhu Guo, and Feiqi Deng
\thanks{Ze-Hao Wu is with School of Mathematics and Big Data, Foshan University, Foshan 528000, China.
        {\tt\small  Email: zehaowu@amss.ac.cn} }
 \thanks{Hua-Cheng Zhou  is with School of Mathematics and Statistics, Central South University, Changsha 410075, China.
 {\tt\small  Email: hczhou@amss.ac.cn} }
 \thanks{Bao-Zhu Guo  is with Academy of Mathematics and Systems Science, Chinese Academy of Sciences, Beijing 100190, China.
 {\tt\small  Email: bzguo@iss.ac.cn} }
 \thanks{Feiqi Deng  is with Systems Engineering Institute, South China University of Technology, Guangzhou 510640, China.
 {\tt\small  Email: aufqdeng@scut.edu.cn}
 }}%

%
%

\markboth{}%
{Shell \MakeLowercase{\textit{et al.}}: Bare Demo of IEEEtran.cls
for Journals}
%



\maketitle

\begin{abstract}
In this paper, the  convergence and noise-tolerant performance of a  tracking differentiator
in the presence of multiple stochastic disturbances are  investigated for the first time.
We consider a quite  general case where  the input signal is  corrupted  by  additive colored noise,
and the tracking differentiator itself  is disturbed by additive colored noise and white noise.
It is shown that the tracking differentiator tracks the input signal and its generalized derivatives in mean square
and even in almost sure sense when the stochastic noise affecting the input signal is vanishing. 
Some numerical simulations are performed to validate the theoretical results.
\end{abstract}


%
\IEEEpeerreviewmaketitle

\section{Introduction}
%
%
%
%

\IEEEPARstart{I}t is generally known that the powerful yet primitive proportional-integral-derivative
(PID) control law developed during the period of the 1920s-1940s has been dominating  control engineering for one century.
However,  the derivative control may be not practically feasible because the classical differentiation is  sensitive to
and may amplify the noise. A  noise-tolerant tracking differentiator (TD)
which is also the first part of the powerful active disturbance rejection control (ADRC) technology \cite{Han2009},
was first proposed by Han in \cite{han1994}. A detailed  comparison with other  differentiators aforementioned
  was made in  \cite{xue2010}. The effectiveness of TD has been validated by  numerous  numerical experiments and 
engineering applications, see, for instance \cite{su2005,tian2013,shen2017,zhang2018}.
The  convergence  of a simple linear TD was first presented in \cite{Guo2002}  with application for  online estimation 
of the frequency of  sinusoidal signals.   Some convergence analyses of the nonlinear TD for both two-dimensional and
   high-dimensional cases under some weak assumptions were given in  \cite{Guo2011}. The  weak convergence of a nonlinear
TD based on finite-time stable system was presented in \cite{Guo2012}.   The more comprehensive introduction
  including the  convergence analysis  of linear, nonlinear  and finite-time stable TD can be found in  Chapter 2
  of the monograph  \cite{Guo2016} without considering input noises. However, in practical implementations, 
  stochastic disturbances are inevitable and the stochastic systems are modelled in many situations, see, 
  for instance  \cite{krtics,Pan2,bask3,ZhaoDeng2022tac}. Motivated from this consideration, 
  in this paper,  we investigate for the first time, the convergence and noise-tolerant performance of   
 TD when the input signal is corrupted  by  additive colored noise, and the TD itself is disturbed by  additive 
 colored  and  white noises.

The main contributions and novelty of this paper are twofold. Firstly, from a theoretical perspective, the
 convergence and noise-tolerant performance of TD are firstly  analyzed  rigorously in the presence of 
 multiple stochastic disturbances which include  both additive colored noise and  white noise.
 Secondly, the theoretical results reveal that the states of TD track both the input signal and its generalized derivatives in
mean square and even in almost sure sense in the case that the stochastic noise corrupting the input signal is vanishing.

We proceed as follows. In the next section, section \ref{Se2}, the problem is  formulated  and
some preliminaries are presented. In Section \ref{Se3}, the  main results are presented with proofs in Appendices.
 Some numerical simulations are presented in section \ref{Se4}, followed up  concluding remarks
 in section \ref{Se5}.

\section{Problem formulation and preliminaries}\label{Se2}
The following notations are used throughout the paper.
The $\mathbb{R}^{n}$ denotes the $n$-dimensional Euclidean space;
$\mathbb{E}X$ or $\mathbb{E}(X)$ denotes the mathematical expectation of a random
variable $X$; For a vector or matrix $X$, $X^{\top}$ represents its transpose;
$|X|$ represents the absolute value of a scalar $X$, and
$\|X\|$ represents the Euclidean norm of a vector $X$;
$a\wedge b$ denotes the minimum of reals $a$ and $b$.

Let $(\Omega,\mathcal{F},\mathbb{F}, P)$ be a complete filtered probability space with a filtration
$\mathbb{F}=\{\mathcal{F}_{t}\}_{t\geq 0}$  on which three  mutually independent one-dimensional 
standard Brownian motions $B_{i}(t)\;(i=1,2,3)$ are defined. In many cases,  the stochastic 
disturbances are modeled by white noise which is a stationary stochastic process that has zero mean and constant
spectral density and  is the generalized derivative of the Brownian motion (see, e.g., \cite[p.51, Theorem 3.14]{duan2015}).
Nevertheless, the white noise does not always  well describe the  stochastic disturbances occurring in
 nature because its $\delta$-function correlation is an idealization of the correlations of real processes 
 which often have finite, or even long, correlation time \cite{colorednoise}.
A more realistic description could be given by an exponentially correlated process,
which is known as colored noise or Ornstein-Uhlenbeck process \cite{colorednoise,colourednoise2}.
Let $w_{i}(t)\;(i=1,2)$ denote the colored noise.
They are the solutions of  the It\^{o}-type stochastic differential equations
(see, e.g., \cite[p.426]{colorednoise}, \cite[p.101]{mao}):
\begin{equation}\label{21equ}
dw_{i}(t)=-\alpha_{i}w_{i}(t)dt+ \alpha_{i}\sqrt{2\beta_{i}}dB_{i}(t),
\end{equation}
where $\alpha_{i}>0$ and $\beta_{i}$ are given constants describing
the correlation time and the noise intensity, respectively,
 and the initial values $w_{i}(0)\in L^{2}(\Omega; \mathbb{R})$ are independent of $B_{i}(t)$.
 In other words, the parameters $\alpha_{i}$ describe the bandwidth of the noise, while $\beta_{i}$ denote
its spectral height, and the correlation functions of the processes $w_{i}(t)$ are
 more realistic exponential functions yet not the  $\delta$-ones (see, e.g., \cite{colorednoise}).
In what follows,  $\alpha_{i}$ and $\beta_{i}$ can be  unknown constants.

Let $v(t)$ be a time-varying input signal which is  supposed to be contaminated by 
 additive colored noise. Therefore, the input signal is actually
\begin{equation}\label{vdefiniton}
v^{*}(t):= v(t)+\sigma_{1}w_{1}(t),
\end{equation}
where $\sigma_{1}$ is a  constant that could be unknown and represents the intensity of 
the colored noise. In addition, we consider a general case where  the system constructing TD is disturbed 
by  both additive colored noise and  white noise as follows:
\begin{equation}\label{TD}
 \left\{\begin{array}{l}
dx_{1}(t)=x_{2}(t)dt, \cr
dx_{2}(t)=x_{3}(t)dt, \cr
\hspace{1.3cm} \vdots \cr
dx_{n-1}(t)=x_{n}(t)dt,\cr
dx_{n}(t)=r^{n}f(x_{1}(t)-v^{*}(t),\frac{x_{2}(t)}{r},\cdots,\frac{x_{n}(t)}{r^{n-1}})dt\cr\hspace{1.1cm}
+\sigma_{2}w_{2}(t)dt+\sigma_{3}dB_{3}(t),
\end{array}\right.
\end{equation}
where $r>0$ is a tuning parameter, $f:\mathbb{R}^{n}\rightarrow \mathbb{R}$  is an appropriate known 
function chosen to satisfy the following Assumption (A1), and ``$\sigma_{2}w_{2}(t)+\sigma_{3}\dot{B}_{3}(t)$"  
represents the multiple stochastic disturbances  with  $\sigma_{i}\;(i=2,3)$ being constants that could  be unknown and
$\dot{B}_{3}(t)$ being the white noise which is the  formal  derivative of the  Browian motion. {  Han's TD  in \cite{han1994}
is a special case of \dref{TD} with $\sigma_i=0, i=1,2,3.$}
The consideration of such a TD  is based on three aspects: First, such a  TD itself in noisy environment is  quite general
whereas the  TD without any noise corruption is just a special case of  $\sigma_{2}=\sigma_{3}=0$.
Second, the  quantization errors caused by the digital implementation of TD always exist
 and can be regarded as a kind of process noise. Finally, TD is the first part
 of the powerful ADRC which has been hardwired into the general purpose control chips made by
industry giants such as Texas Instruments \cite{Ti}, where the hardware might work
in noisy environment.

In addition, it should be noticed that the solution of (\ref{TD}) depends  on the
tuning parameter $r$. Hereafter, we always drop $r$ from  solutions by abuse of notation without confusion.

The following Assumption (A1) is  a prior assumption about the known function
$f(\cdot)$ chosen in   constructing TD \dref{TD}.

{\bf Assumption (A1).}  The $f:\mathbb{R}^{n}\rightarrow \mathbb{R}$ is a locally Lipschitz 
continuous function with respect to its arguments, $f(0,\cdots,0)=0$, and there exist known 
constants $\lambda_{i}>0\;(i=1,2,3,4)$ and  a twice continuously differentiable function
$V: \mathbb{R}^{n}\to [0,\infty)$ which is positive definite
and radially unbounded such that
\begin{eqnarray}
\disp
&&\lambda_{1}\|z\|^{2}\leq V(z)\leq \lambda_{2}\|z\|^{2},\; \lambda_{3}\|z\|^{2}\leq W(z)\leq \lambda_{4}\|z\|^{2},
\cr && \sum^{n-1}_{i=1}\frac{\partial V(z)}{\partial
z_{i}}z_{i+1}+\frac{\partial V(z)}{\partial
z_{n}}f(z)\leq -W(z), \cr&&
\left|\frac{\partial V(z)}{\partial
z_{j}}\right|\leq c_{1}\|z\|,\; \left|\frac{\partial^{2} V(z)}{\partial
z^{2}_{j}}\right|\leq c_{2},\;\; \cr &&
\forall\;z=(z_{1},z_{2},\cdots,z_{n})\in
\mathbb{R}^{n},\; j=1,n,
\end{eqnarray}
 for some nonnegative continuous function $W:\mathbb{R}^{n}\rightarrow [0,\infty)$  and some constants $c_{i}>0\;(i=1,2)$.

\begin{remark}
Generally speaking, the Assumption (A1) guarantees that the function $f:\mathbb{R}^{n}\to  \mathbb{R}$ is chosen so that
the zero equilibrium state of the following system
\begin{equation}\label{2900}
\dot{z}(t)=\left(z_{2}(t),z_{3}(t),\cdots, f(z(t)\right))
\end{equation}
is globally exponentially stable with $z=(z_1,z_2,\cdots,z_{n})$.
It is easy to verify that the simplest example to satisfy Assumption (A1) is the linear function
\begin{equation}\label{linearfunction}
f(z)=a_{1}z_{1}+\cdots+a_{n}z_{n},
\end{equation}
where the parameters $a_{i}\;(i=1,2,\cdots,n)$  are chosen such that the matrix
\begin{equation}\label{matric2}\disp
A=
\begin{pmatrix}
0     & 1      &0      &  \cdots   &  0             \cr 0     & 0 &
1     &  \cdots   &  0             \cr \cdots &\cdots & \cdots &
\cdots& \cdots  \cr 0     & 0     & 0     &  \ddots   &  1 \cr a_1
& a_2 &\cdots & a_{n-1} & a_n
\end{pmatrix}_{n\times n}
\end{equation}
is Hurwitz.  The TD with linear function $f(\cdot)$ given by (\ref{linearfunction}) is referred as linear TD in what follows.
\end{remark}

 The solution of (\ref{21equ}) can be explicitly expressed as
\begin{eqnarray}
 w_{i}(t)=e^{-\alpha_{i}t}w_{i}(0)+\int^{t}_{0}e^{-\alpha_{i}(t-s)}\alpha_{i}\sqrt{2\beta_{i}}dB_{i}(s).
\end{eqnarray}
Define
\begin{equation}
\gamma_{i}=\mathbb{E}|w_{i}(0)|^{2}+\alpha_{i}\beta_{i},\; i=1,2.
\end{equation}
By the It\^{o} isometric formula, it is easy to verify  that
the  second moments of $w_{i}(t)\; (i=1,2)$ are bounded:
\begin{eqnarray}\label{11boundedness}
\disp &&\mathbb{E}|w_{i}(t)|^{2} \cr&&
= e^{-2\alpha_{i}t}\mathbb{E}|w_{i}(0)|^{2}
+\mathbb{E}|\int^{t}_{0}e^{-\alpha_{i}(t-s)}\alpha_{i}\sqrt{2\beta_{i}}dB_{i}(s)|^{2} \cr && \leq
\mathbb{E}|w_{i}(0)|^{2}+2\alpha^{2}_{i}\beta_{i}\int^{t}_{0}e^{-2\alpha_{i}(t-s)}ds \cr &&
\leq \gamma_{i}, \; \forall t\geq 0.
\end{eqnarray}
This  is the reason behind that  the TD may be feasible when the input signal is disturbed by
additive colored noise.

\section{Main results}\label{Se3}

Set $\hat{B}_{1}(t)=\sqrt{r}B_{1}\left(\frac{t}{r}\right),\; \hat{B}_{3}(t)=\sqrt{r}B_{3}\left(\frac{t}{r}\right).$
Notice that for any $r>0$, $\hat{B}_{1}(t)$ and $\hat{B}_{3}(t)$ are still
mutually independent one-dimensional standard Brownian motions.
 By  definition of $v^{*}(t)$ in (\ref{vdefiniton}), it follows  that
\begin{equation}
dv^{*}(t)=\dot{v}(t)dt-\sigma_{1}\alpha_{1}w_{1}(t)dt+\sigma_{1}\alpha_{1}\sqrt{2\beta_{1}}dB_{1}(t),
\end{equation}
and then
\begin{equation}
dv^{*}(\frac{t}{r})=\dot{v}(\frac{t}{r})dt-
\frac{\sigma_{1}\alpha_{1}}{r}w_{1}(\frac{t}{r})dt+\frac{\sigma_{1}\alpha_{1}\sqrt{2\beta_{1}}}{\sqrt{r}}d\hat{B}_{1}(t),
\end{equation}
where $\dot{v}(\frac{t}{r})$
denotes, in what follows,  the derivative of $v(\frac{t}{r})$ with respect to the time  $t$.
For $i=2,\cdots,n$, set
\begin{eqnarray}\label{variabled}
 y_{1}(t)=x_{1}(\frac{t}{r})-v^{*}(\frac{t}{r}), \;
 y_{i}(t)=\frac{1}{r^{i-1}}x_{i}(\frac{t}{r}).
\end{eqnarray}
A direct computation shows that $y(t)=(y_{1}(t),\cdots,y_{n}(t))$ satisfies the following
It\^{o}-type stochastic differential equation:
\begin{equation}\label{kerequ}
\left\{\begin{array}{l}\disp
dy_{1}(t)=y_{2}(t)dt-\dot{v}(\frac{t}{r})dt+\frac{\sigma_{1}\alpha_{1}}{r}w_{1}(\frac{t}{r})dt
\cr-\frac{\sigma_{1}\alpha_{1}\sqrt{2\beta_{1}}}{\sqrt{r}}d\hat{B}_{1}(t), \cr
dy_{2}(t)=y_{3}(t)dt, \cr \hspace{1.2cm} \vdots \cr
dy_{n-1}(t)=y_{n}(t)dt, \cr
dy_{n}(t)=f(y(t))dt+\frac{\sigma_{2}}{r^{n}}w_{2}(\frac{t}{r})dt+\frac{\sigma_{3}}{r^{n-\frac{1}{2}}}d\hat{B}_{3}(t).
\end{array}\right.
\end{equation}

We first introduce  Lemma \ref{lemmasef} below to present the existence and 
uniqueness of the global solution to system  (\ref{kerequ}) and
 give an estimate of  the second moment of the global solution.
\begin{lemma}\label{lemmasef}
Suppose that $v:[0,\infty)\rightarrow \mathbb{R}$ is a continuously differentiable function 
satisfying $\sup_{t\geq 0}(|v(t)|+|\dot{v}(t)|)\leq M$ for some
constant $M>0$  and Assumption (A1) holds, and
the tuning parameter $r$ is chosen so that  $r\geq 1$. Then, for any initial value $x(0)\in \mathbb{R}^{n}$,
system (\ref{kerequ}) admits a unique global solution which
satisfies
\begin{eqnarray}\label{fde4}
 \mathbb{E}\left(\sup_{0\leq s\leq t}\|y(s)\|^{2}\right)\leq
 \frac{1}{\lambda_{1}}\left(N_{1}+\frac{N_{2}}{N_{3}}\right)e^{N_{3}t}, \;\forall t\geq 0,
\end{eqnarray}
where the constants $N_{i}(i=1,2,3)$
are specified in (\ref{constantsds}).
\end{lemma}

\begin{proof}
See  ``Proof of Lemma \ref{lemmasef}'' in  Appendix A.
\end{proof}

\hspace{0.5cm}

In what follows, a value range of the tuning parameter
to guarantee the convergence of TD (\ref{TD}) can be specified as
\begin{equation}\label{37df}
R_{0}:=\left\{r\geq 1:\frac{1}{r}+\frac{1}{2r^{2n-1}}\leq\frac{\theta\lambda_{3}}{\lambda_{2}}\right\},
\end{equation}
where $\theta\in(0,1)$ is any chosen parameter. Note that when  $\theta\in(0,1)$ is increasing, the range $R_{0}$ will increase
as well  but
the exponential decay rate $\frac{(1-\theta)\lambda_{3}}{\lambda_{2}}$ { associated with the tracking error}
would be reduced.

The convergence result of TD (\ref{TD}) in the presence of multiple
stochastic disturbances is summarized as the following Theorem \ref{theorem}.

\begin{theorem}\label{theorem}
Suppose that $v:[0,\infty)\rightarrow \mathbb{R}$ is a continuously differentiable function 
satisfying $\sup_{t\geq 0}(|v(t)|+|\dot{v}(t)|)\leq M$ for some
constant $M>0$  and Assumption (A1) holds. Then, for any tuning parameter $r\in R_{0}$,  
initial value $x(0)\in \mathbb{R}^{n}$ and  $T>0$, the TD (\ref{TD}) admits a unique global solution satisfying
\begin{eqnarray}\label{errord}
\hspace{-0.4cm}\mbox{(i)}\;\;\;\;\mathbb{E}|x_{1}(t)-v(t)|^{2}\leq \frac{(1+\frac{1}{\mu})\Gamma}{r}+(1+\mu)\sigma^{2}_{1}\gamma_{1}
\end{eqnarray}
uniformly in $t\in [T,\infty)$, where $\mu$ is any positive constant and $\Gamma$ 
specified in (\ref{734fd}) is a positive constant   independent of $r$;
\begin{eqnarray}\label{error2}
\hspace{-2.2cm}\mbox{(ii)}\;\;\;\limsup_{r \to \infty} \mathbb{E}|x_{1}(t)-v(t)|^{2}\leq \sigma^{2}_{1}\gamma_{1}
\end{eqnarray}
uniformly in $t\in [T,\infty)$;
\begin{eqnarray}\label{319fd}
\hspace{-1.5cm}\mbox{(iii)}\;\lim_{r \to \infty}|x_{1}(t)-v(t)|=0 \;\; \mbox{almost surely}
\end{eqnarray}
uniformly in $t\in [T,\infty)$ when   $\sigma_{1}=0$.
\end{theorem}
\begin{proof}
See  ``Proof of Theorem \ref{theorem}" in Appendix B.
\end{proof}

\hspace{0.5cm}

\begin{remark}
Note that the tracking error system (\ref{kerequ}) is an It\^{o}-type stochastic system. Thus, the convergence in mean square
sense is natural   because the  It\^{o} integral terms are
zero as   martingales after taking mathematical expectation. In addition, the mean square sense denotes
the convergence of an average level of the tracking error, which could be  in line with
engineering applications since the deviation of every sample path of the tracking error from the average level is
often small. Finally,  we can see from  (\ref{errord}) that the upper bound of the 
tracking error in mean square can approach $\sigma^{2}_{1}\gamma_{1}$
 arbitrarily and quickly  by tuning the parameter $r$ to be sufficiently large since the convergence 
 time $T$ is any positive constant. {  This is  what we mean by  ``TD
is not sensitive to  input noise''.  It seems impossible to make the
tracking error as small as possible when $\sigma_1\neq 0$ in \dref{vdefiniton}.}
\end{remark}

\begin{remark}
It is noteworthy that the selection of the function $f(\cdot)$  guarantees that  
the ``nominal part" of the tracking error system (\ref{kerequ}) defined
in (\ref{2900}) is exponentially stable with the decay rate   $\frac{\lambda_{3}}{\lambda_{2}}$ 
which depends on $f(\cdot)$. By
 definition of $\Gamma$ { from (\ref{734fd}) in Appendix B,}
the constant $\Gamma$ which is a part of the tracking error (depending on $f(\cdot)$) becomes smaller
if  the decay rate $\frac{\lambda_{3}}{\lambda_{2}}$ becomes larger.
 In addition, another advantage could be mentioned is that  when  the decay rate $\frac{\lambda_{3}}{\lambda_{2}}$ becomes larger,
  the value range of the tuning parameter $r$  defined in (\ref{37df}) will  increase.
\end{remark}

Finally, we indicate an  important fact that $x_{i}(t)\;(i=2,3,\cdots,n)$ can always 
be regarded as an approximation of the corresponding $(i-1)-$th derivative of $v(t)$ in terms of generalized
derivative whatever the classical derivatives of $v(t)$ exist or not.
In fact, for any $a>0$, let $C^{\infty}_{0}(0,a)$ be the set that contains all
infinitely differentiable functions with compact support on $(0,a)$.
Remember that for any  locally integrable function $h:(0,a)\rightarrow \mathbb{R}$,
the usual $(i-1)$-th generalized
derivative of $h$, still denoted by $h^{(i-1)}$, always exists in the sense of distribution
defined as a functional on $C^{\infty}_{0}(0,a)$ that
\begin{equation}\label{definitiondervc}
h^{(i-1)}(\varphi)=(-1)^{i-1}\int^{a}_{0}h(t)\varphi^{(i-1)}(t)dt,
\end{equation}
for every test function $\varphi\in C^{\infty}_{0}(0,a)$ and $2\leq i\leq n$ (see, e.g., \cite[p.43]{Guo2016}).
In addition, a generalized stochastic
process $\Phi$ is simply a random generalized function in the following sense:
For every test function $\varphi\in C^{\infty}_{0}(0,a)$, a random variable $\Phi(\varphi)$
is assigned such that the functional $\Phi$ on $C^{\infty}_{0}(0,a)$ is linear and continuous
(see, e.g., \cite[p.50]{duan2015}).
Thus, for each $i=2,3,\cdots,n$, $x_{i}$ itself can be regarded as a generalized stochastic
 process in the sense that
\begin{equation}\label{38fd}
x_{i}(\varphi)=\int^{a}_{0}x_{i}(t)\varphi(t)dt,\; \forall \varphi\in C^{\infty}_{0}(0,a).
\end{equation}

 For each $i=2,3,\cdots,n$, the state $x_{i}$ of  the TD (\ref{TD})  is convergent
to the $(i-1)$-th generalized derivative of the input signal $v$ in mean square
and almost sure sense, which is summarized in the following Theorem \ref{weakconvergence}.

\begin{theorem}\label{weakconvergence}
Suppose that $v:[0,a]\rightarrow \mathbb{R}$ is  continuously differentiable   and Assumption (A1) holds. Then,
for any initial value $x(0)\in \mathbb{R}^{n}$, the TD (\ref{TD}) admits a unique global solution, and
for any $\varphi\in C^{\infty}_{0}(0,a)$ and  all  $i=2,3,\cdots,n$, there holds
\begin{eqnarray}\label{315fd}
\hspace{-2.5cm}\mbox{(i)}\;\;\;\limsup_{r\to \infty} \mathbb{E}|x_{i}(\varphi)-v^{(i-1)}(\varphi)|^{2}
\cr \hspace{-1.6cm}\;\;\; \leq  a^{2}\sup_{t\in(0,a)}|\varphi^{(i-1)}(t)|^{2}\sigma^{2}_{1}\gamma_{1};
\end{eqnarray}
\begin{eqnarray}\label{316d}
\hspace{-0.4cm}\mbox{(ii)}\;\lim_{r\to \infty}|x_{i}(\varphi)-v^{(i-1)}(\varphi)|=0 \;\; \mbox{almost surely}
\end{eqnarray}
  when the additive colored noise affecting the input signal is vanishing, i.e., $\sigma_{1}=0$.
\end{theorem}

\begin{proof}
See `` Proof of Theorem \ref{weakconvergence}'' in Appendix C.
\end{proof}

\hspace{0.5cm}

\begin{remark}
The   convergence of linear TD without requiring Assumption (A1) can be concluded directly
from Theorems \ref{theorem} and  \ref{weakconvergence}.
This is because the matrix $A$ in (\ref{matric2}) defined by the designed parameters 
$a_{i}\;(i=1,2,\cdots,n)$  is Hurwitz so that there exists a unique
 positive definite matrix solution $Q$ to the
Lyapunov equation $QA+A^\top Q=-I_{n\times n}$ for $n$-dimensional
identity matrix $I_{n\times n}$. For this reason, we can
define the Lyapunov functions $V:\mathbb{R}^{n}\to [0,\infty)$  and
$W:\mathbb{R}^{n}\to [0,\infty)$ by
$V(z)=zQz^\top $  and $W(z)=\|z\|^2$ for $z \in \mathbb{R}^{n}$, respectively.
It is then easy to verify that all conditions
in Assumption (A1)  are satisfied, where the parameters in
Assumptions (A1) are specified as
$\lambda_{1}=\lambda_{\min}(Q)$, $\lambda_{2}=\lambda_{\max}(Q)$,
$\lambda_{3}=\lambda_{4}=1$, $c_{1}=c_{2}=2\lambda_{\max}(Q)$,
with $\lambda_{\min}(Q)$ and $\lambda_{\max}(Q)$ being  respectively
the minimal   and maximal eigenvalues of the matrix $Q$.
\end{remark}

\begin{remark}
The present  paper focuses only on  the convergence and noise-tolerant performance for { TD} in the presence
of multiple stochastic disturbances. However, in practical applications,  there may exist
 phase lags because of using the integration  of TD to estimate the  derivatives of the input signal,
which can  be overcome by  introducing feedforward in the design of TD (\cite{tian2013}).
\end{remark}

\section{Numerical Simulations}\label{Se4}
In this section, some numerical simulations are presented to
verify the effectiveness  of the main results. Let the
input signal be $v(t)=\sin(3t+1)$. We design a second-order linear TD
and a second-order nonlinear TD in the form of (\ref{TD}) in the presence of 
multiple stochastic disturbances.
The linear TD is produced by a linear function given by
\begin{equation}\label{linearfunctiofgn}
f(z_{1},z_{2})=-2z_{1}-4z_{2},\; \forall (z_{1},z_{2})\in \mathbb{R}^{2}.
\end{equation}
Motivated by the nonlinear feedback controller design, a nonlinear function used for the
 construction of nonlinear TD can be the linear function (\ref{linearfunctiofgn})  adding
with a Lipschitz continuous function given by
\begin{equation}
f(z_{1},z_{2})=-2z_{1}-4z_{2}-\phi(z_{1}),\; \forall (z_{1},z_{2})\in \mathbb{R}^{2},
\end{equation}
where
\begin{equation}
\phi(s)= \left\{\begin{array}{ll}\disp -\frac{1}{4\pi}, &
s\in(-\infty,-\frac{\pi}{2}) , \crr\disp \frac{1}{4\pi}\sin s,&
s\in[-\frac{\pi}{2},\frac{\pi}{2}],\crr\disp \frac{1}{4\pi},& s\in (\frac{\pi}{2},+\infty),
\end{array}\right.
\end{equation}
and the Assumption (A1) holds for (\cite[p.196]{Guo2016})
\begin{eqnarray}
&&V(z_{1},z_{2})=1.375z^{2}_{1}+0.1875z^{2}_{2}+0.5z_{1}z_{2}, \cr &&
W(z_{1},z_{2})=0.5z^{2}_{1}+0.5z^{2}_{2},
\lambda_{1}=0.13, \lambda_{2}=1.43, \cr&& \lambda_{3}=\lambda_{4}=0.5,c_{1}=3.91, c_{2}=2.75.
\end{eqnarray}

In Figures \ref{figure1}-\ref{figure3}, some relative parameters  are chosen as
\begin{eqnarray}
\alpha_{1}=\alpha_{2}=3, \beta_{1}=\beta_{2}=\frac{1}{18}, \sigma_{1}=0.2, \;\sigma_{2}=\sigma_{3}=2,
\end{eqnarray}
the initial values are taken  as
\begin{eqnarray}
x_{1}(0)=\sin(1), x_{2}(0)=0, w_{1}(0)=1, w_{2}(0)=-1,
\end{eqnarray}
the sampling period $\Delta t=0.001$, and
 the diffusion terms $dB_{i}(t)\;(i=1,2,3)$ are simulated by $\sqrt{\Delta t}$ 
 multiplied by random sequences generated by the Matlab program command ``randn".
  The selection of $r$ can be specified by (\ref{37df}). In Figures \ref{figure1}-\ref{figure2} and  \ref{figure3},
 we choose $r=30$ and $r=15$, respectively, and it can be easily verified for the
nonlinear case that if $r=15$, $\frac{1}{r}+\frac{1}{2r^{2n-1}}=\frac{1}{15}+\frac{1}{2\times 15^{3}}\approx0.07<
 \frac{\theta\lambda_{3}}{\lambda_{2}}\approx 0.17$, i.e., $r=15\in R_{0}$ and then $r=30\in R_{0}$, where we set $\theta=0.5$.

It is seen from Figure \ref{figure1}  that
the states $x_{1}(t)$ and $x_{2}(t)$ of the linear TD track quickly the input signal $v(t)=\sin{(3t+1)}$
and the derivative of the input signal, respectively. It is also observed from Figure \ref{figure2}
that the tracking effect of the nonlinear TD is at least as good as the  linear TD.
These are consistent with the theoretical result that the  tracking error becomes small 
after any given  time $T>0$ with the choice of an appropriate tuning parameter $r$.
Since in Figure \ref{figure3} the tuning parameter $r$ is  reduced to be $r=15$,
it can be seen that the tracking accuracy of the nonlinear TD is relatively not as good as  Figure  \ref{figure2},
  which is consistent with the theoretical result that
the upper bound of the tracking error is inverse proportional  to the tuning parameter $r$. 
In addition, it can be observed that the peaking value
phenomenon does not exist in  Figures \ref{figure1}-\ref{figure3}.

\begin{figure}[!htb]\centering
\subfigure[]
 {\hspace{0cm}\includegraphics[width=8.5cm,height=5.2cm]{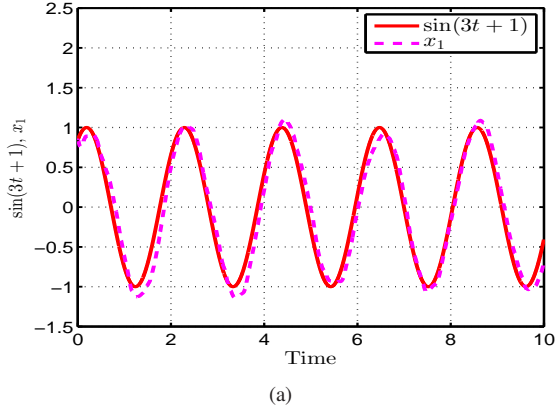}}
 \subfigure[]
 {\hspace{0cm}\includegraphics[width=8.5cm,height=5.2cm]{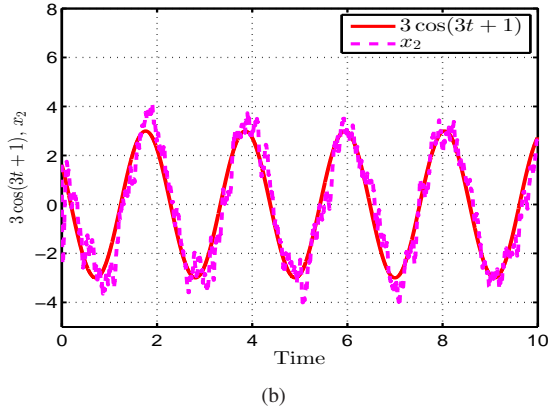}}
 \caption{\normalsize The tracking effect of linear TD with $r=30$.}\label{figure1}
\end{figure}
\begin{figure}[!htb]\centering
\subfigure[]
 {\hspace{0cm}\includegraphics[width=8.5cm,height=5.2cm]{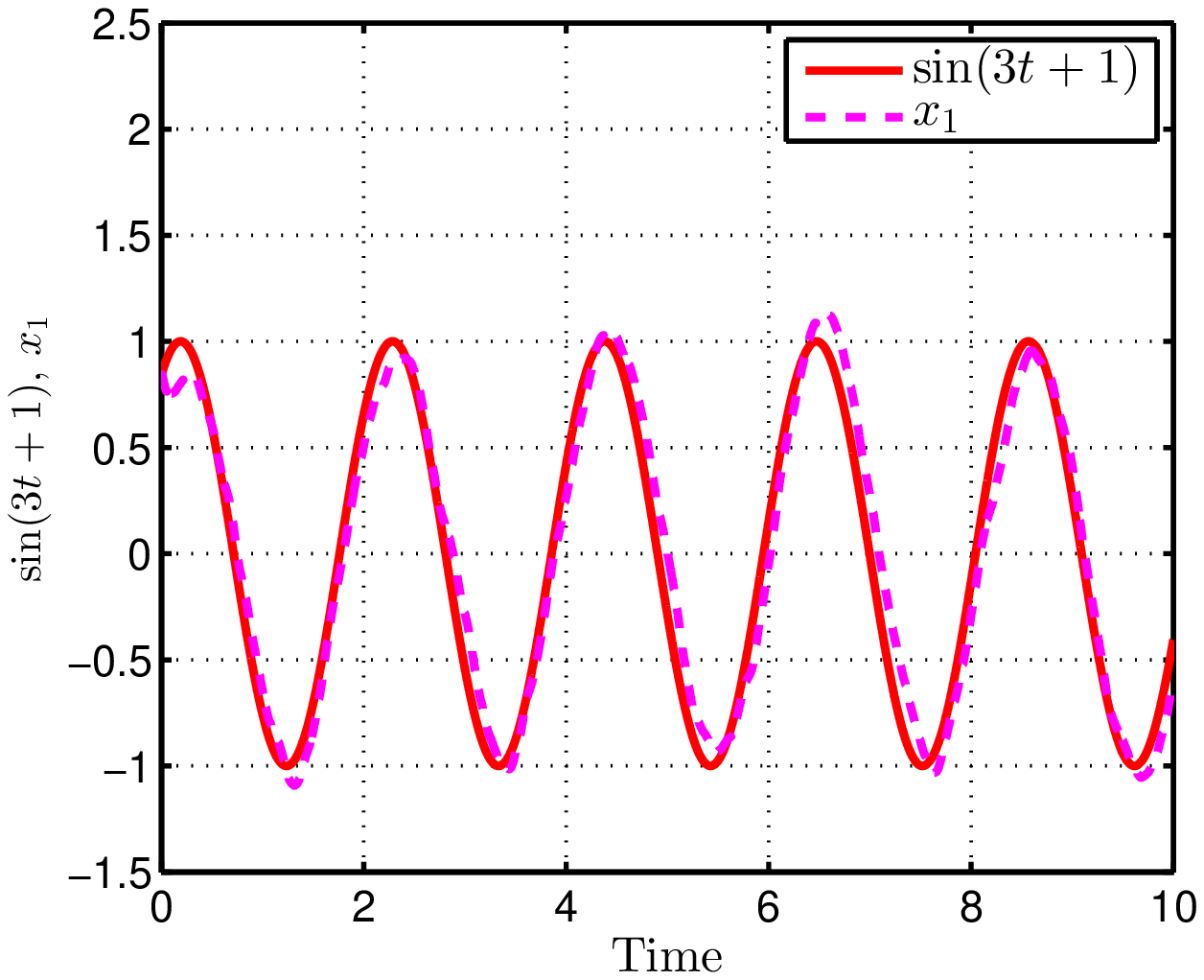}}
 \subfigure[]
 {\hspace{0cm}\includegraphics[width=8.5cm,height=5.2cm]{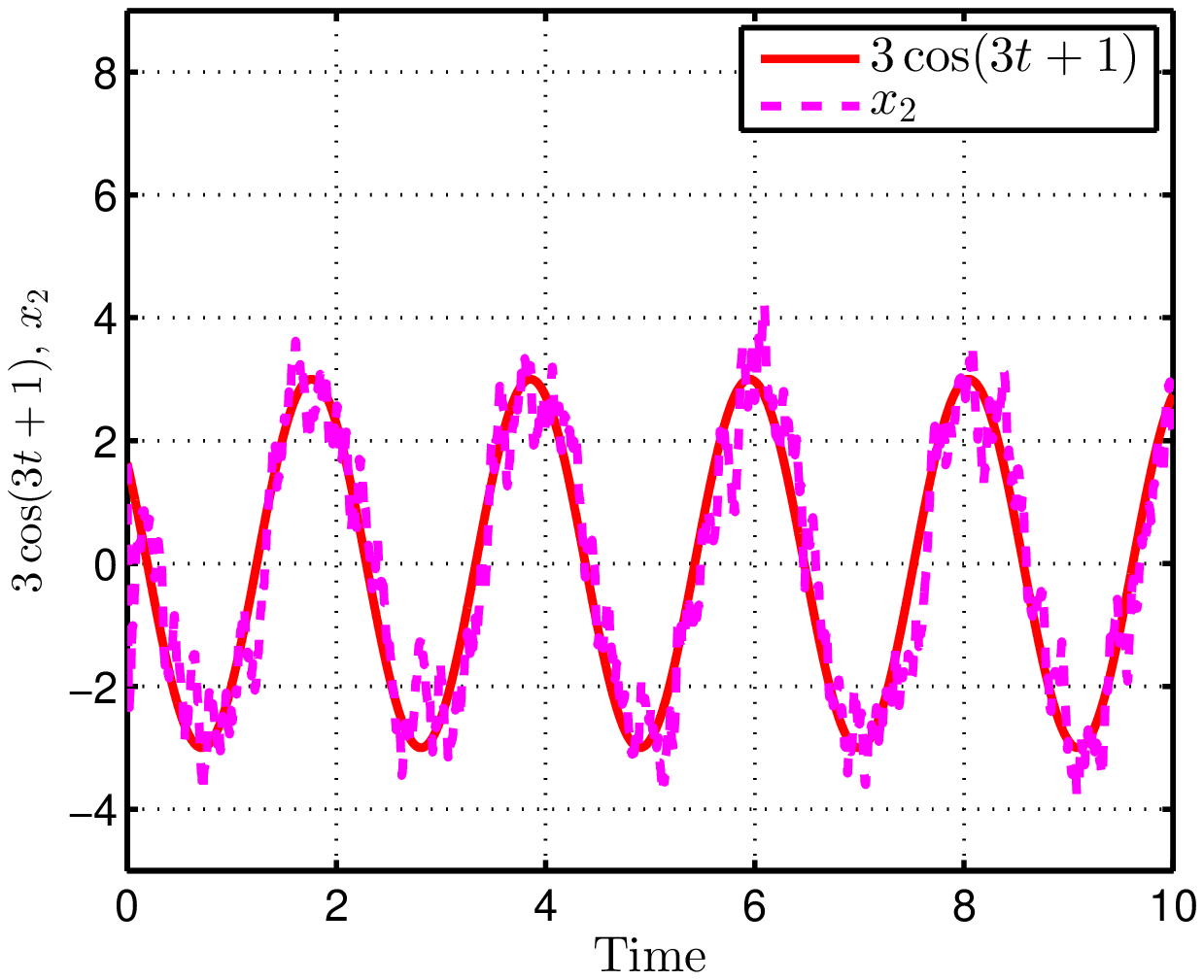}}
\caption{\normalsize The tracking effect of nonlinear TD with $r=30$.}\label{figure2}
\end{figure}
\begin{figure}[!htb]\centering
\subfigure[]
 {\hspace{0cm}\includegraphics[width=8.5cm,height=5.2cm]{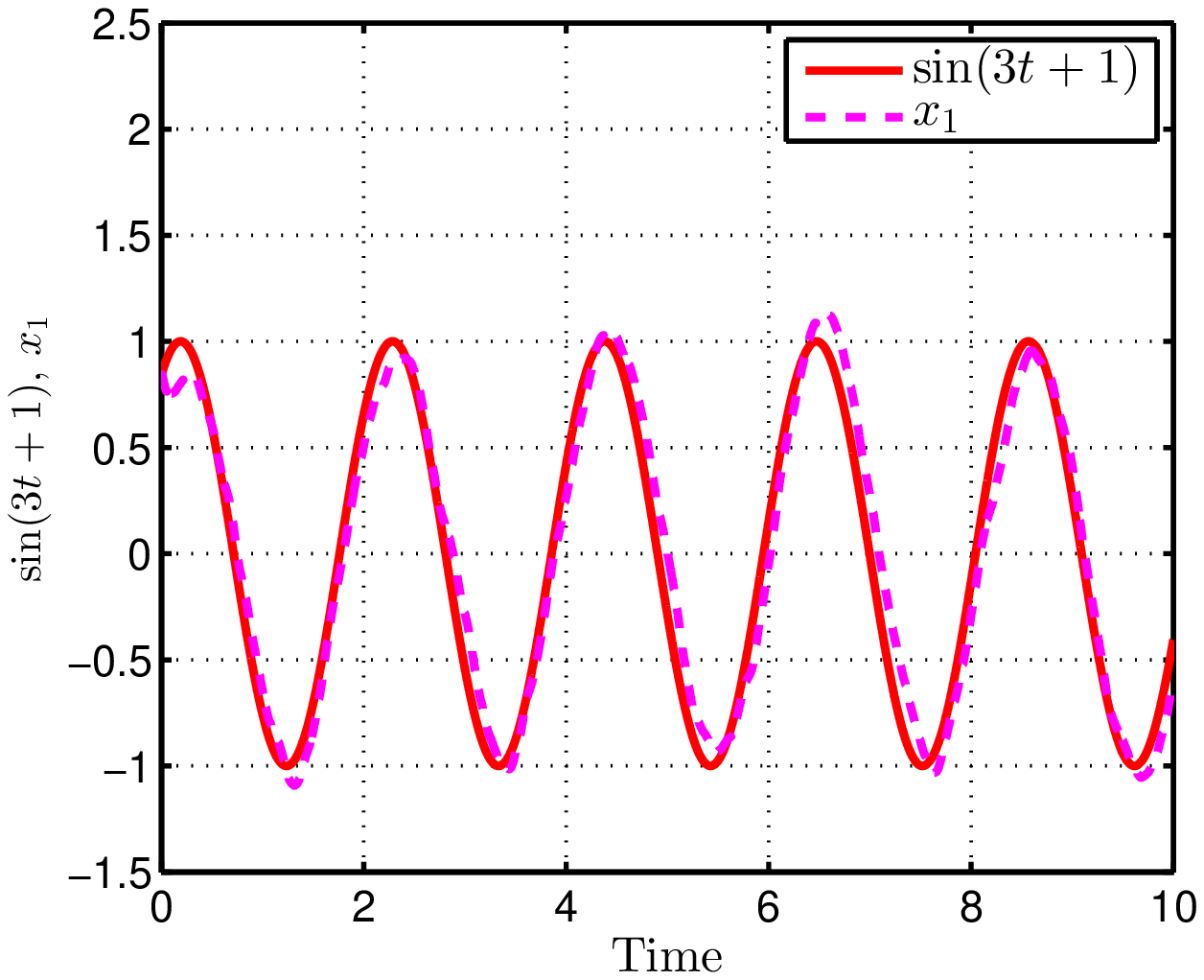}}
\subfigure[]
 {\hspace{0cm}\includegraphics[width=8.5cm,height=5.2cm]{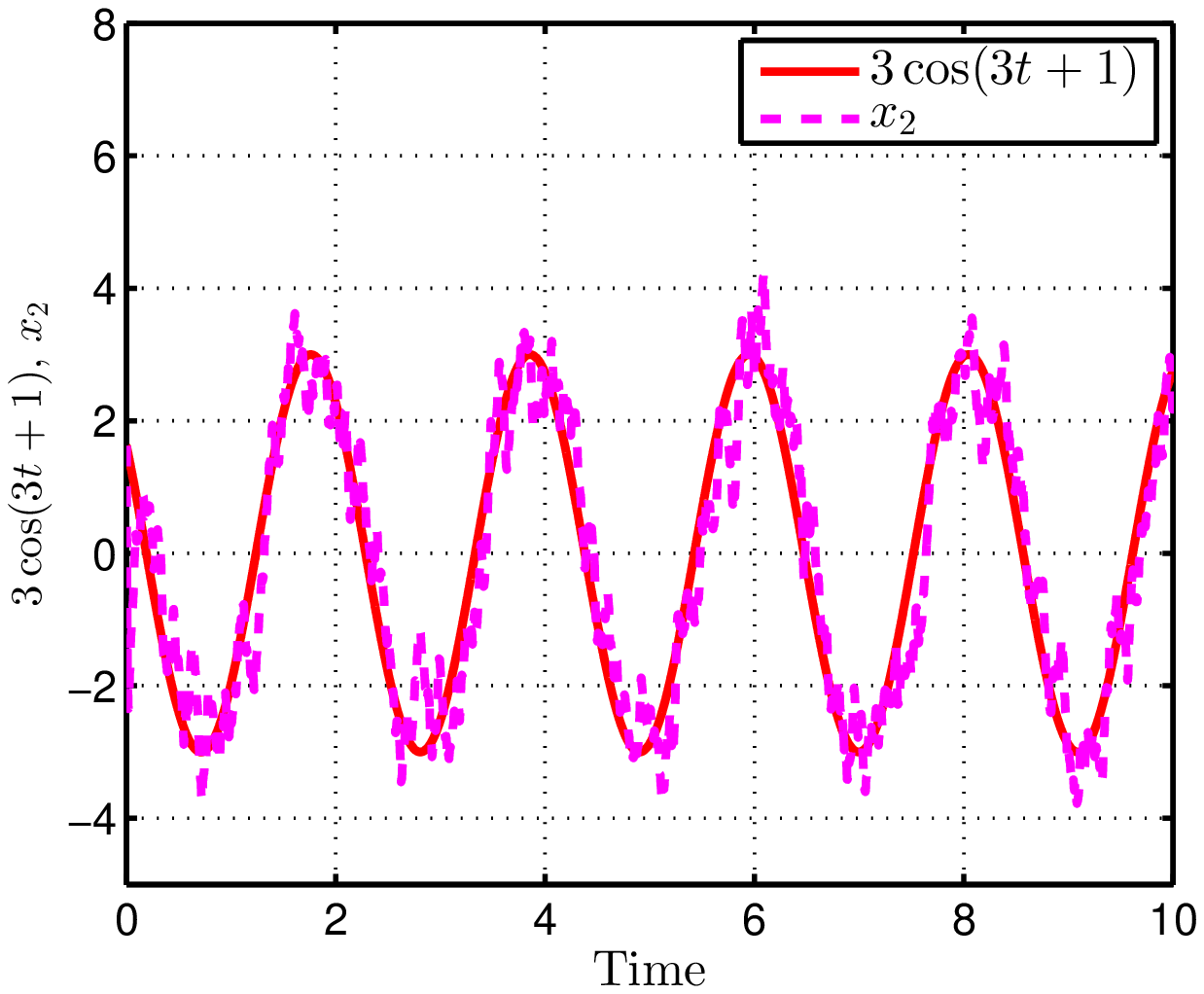}}
\caption{\normalsize The tracking effect of nonlinear TD with $r=15$.}\label{figure3}
\end{figure}

\section{Concluding remarks}\label{Se5}
The convergence and noise-tolerant performance of a tracking differentiator (TD)  are  investigated,
where a general case is considered that the input signal and the TD itself are disturbed  by  additive colored noise
and  additive colored and white noises, respectively. The mathematical  proofs are presented to show that the tracking errors of
the states of TD to the input signal and its generalized derivatives are  convergent in mean square
 and even in almost sure  sense for the special vanishing input noise.
Some numerical simulations are presented to demonstrate the validity of the proposed TD.
Finally, it is worth mentioning that the noise intensity maximum tolerance analysis of TD
 would be a potential interesting problem to be further investigated in the future study.

\begin{center}
APPENDIX A: Proof of Lemma \ref{lemmasef}
\end{center}


 By (\ref{21equ}), $w_{i}(t)\;(i=1,2)$ can be regarded as the augmented state variables of system  (\ref{kerequ}).
Since the function $f(\cdot)$ satisfies the local Lipschitz condition, it follows from the 
existence-and-unique theorem for It\^{o}-type stochastic systems (see, e.g.,\;\cite[p.58,Theorem 3.6]{mao})
 that there exist a unique maximal local solution $y(t)$ over  $t\in [0,\tau)$ where $\tau$ is the explosion time. 
 To show that  $y(t)$ is a global solution, we only  need to show $\tau=\infty$ almost surely.  
 For every integer $m\geq 1$, define the stopping time
$\tau_{m}= \tau\wedge\inf\{t: 0\leq t<\tau, \|y(t)\|\geq m\}$,
and set $\inf\emptyset =\infty$. Since $\{t: 0\leq t<\tau, \|y(t)\|\geq m+1\}\subset \{t: 0\leq t<\tau, \|y(t)\|\geq m\}$,
we have $\tau_{m}\uparrow \tau$ almost surely as $m\rightarrow \infty$. By the It\^{o}'s formula, it yields
\begin{eqnarray}\label{198f}
\disp &&\hspace{-0.7cm} V(y(t))
=V(y(0)) \cr\disp && \hspace{-0.7cm} + \int^{t}_{0}\left[\sum^{n-1}_{i=1}\frac{\partial V(y(s))}{\partial
y_{i}}y_{i+1}(s)+\frac{\partial V(y(s))}{\partial
y_{n}}f(y(s))\right]ds \cr \disp && \hspace{-0.7cm}+\int^{t}_{0}\left[\frac{\partial V(y(s))}{\partial
y_{1}}(-\frac{1}{r}\frac{dv(u)}{du}|_{u=\frac{s}{r}}+\frac{\sigma_{1}\alpha_{1}}{r}w_{1}(\frac{s}{r}))\right]ds \cr\disp &&\hspace{-0.8cm}
+\frac{\sigma^{2}_{1}\alpha^{2}_{1}\beta_{1}}{r}\int^{t}_{0}\frac{\partial^{2} V(y(s))}{\partial y^{2}_{1}}ds
+\frac{\sigma_{2}}{r^{n}}\int^{t}_{0}\frac{\partial V(y(s))}{\partial
y_{n}}w_{2}(\frac{s}{r})ds\cr\disp &&\hspace{-0.8cm}+\frac{\sigma^{2}_{3}}{2r^{2n-1}}\int^{t}_{0}
\frac{\partial^{2} V(y(s))}{\partial y^{2}_{n}}ds-\int^{t}_{0}\frac{\sigma_{1}\alpha_{1}\sqrt{2\beta_{1}}}{\sqrt{r}}\frac{\partial V(y(s))}{\partial
y_{1}}d\hat{B}_{1}(s)\cr\disp &&\hspace{-0.8cm}+\int^{t}_{0}\frac{\sigma_{3}}{r^{n-\frac{1}{2}}}\frac{\partial V(y(s))}{\partial
y_{n}}d\hat{B}_{3}(s).
\end{eqnarray}
Thus, it follows from Assumption (A1), (\ref{11boundedness}) and Young's inequality that
\begin{eqnarray}\label{fdf20}
\disp &&\hspace{-0.6cm} \mathbb{E}\left(\sup_{0\leq u\leq t}V(y(u\wedge \tau_{m}))\right)
\cr\disp && \hspace{-0.6cm}\leq \mathbb{E}V(y(0))
+\mathbb{E}\sup_{0\leq u\leq t}\int^{u\wedge\tau_{m}}_{0}\left[\frac{1}{2\lambda_{1}r}V(y(s))+\frac{c^{2}_{1}M^{2}}{2r}
\right. \cr\disp &&\hspace{-0.6cm} \left.+\frac{1}{2\lambda_{1}r}V(y(s))+\frac{c^{2}_{1}\sigma^{2}_{1}\alpha^{2}_{1}}{2r}|w_{1}(\frac{s}{r})|^{2}+
\frac{c_{2}\sigma^{2}_{1}\alpha^{2}_{1}\beta_{1}}{r}\right. \cr\disp &&\hspace{-0.6cm}
 \left.+\frac{1}{2\lambda_{1}r^{n}}V(y(s))
 +\frac{c^{2}_{1}\sigma^{2}_{2}}{2r^{n}}|w_{2}(\frac{s}{r})|^{2}+\frac{c_{2}\sigma^{2}_{3}}{2r^{2n-1}}\right]ds
\cr\disp &&\hspace{-0.6cm}+\mathbb{E}\left(\sup_{0\leq u\leq t}\int^{u\wedge\tau_{m}}_{0}
-\frac{\sigma_{1}\alpha_{1}\sqrt{2\beta_{1}}}{\sqrt{r}}\frac{\partial V(y(s))}{\partial
y_{1}}d\hat{B}_{1}(s)\right)\cr\disp &&\hspace{-0.6cm}
+\mathbb{E}\left(\sup_{0\leq u\leq t}\int^{u\wedge\tau_{m}}_{0}\frac{\sigma_{3}}{r^{n-\frac{1}{2}}}\frac{\partial V(y(s))}{\partial
y_{n}}d\hat{B}_{3}(s)\right)\cr\disp &&\hspace{-0.6cm}
\leq \mathbb{E}V(y(0))+\int^{t}_{0}\left[\frac{c^{2}_{1}M^{2}}{2r}+\frac{c^{2}_{1}\sigma^{2}_{1}\alpha^{2}_{1}\gamma_{1}}{2r}
+\frac{c_{2}\sigma^{2}_{1}\alpha^{2}_{1}\beta_{1}}{r}
\right.\cr\disp && \hspace{-0.6cm}\left.+\frac{c^{2}_{1}\sigma^{2}_{2}\gamma_{2}}{2r^{n}}+\frac{c_{2}\sigma^{2}_{3}}{2r^{2n-1}}
+(\frac{1}{\lambda_{1}r}+\frac{1}{2\lambda_{1}r^{n}})\mathbb{E}V(y(s\wedge\tau_{m}))\right]ds
\cr\disp &&\hspace{-0.6cm}
+\mathbb{E}\left(\sup_{0\leq u\leq t}\int^{u\wedge\tau_{m}}_{0}-\frac{\sigma_{1}\alpha_{1}\sqrt{2\beta_{1}}}{\sqrt{r}}\frac{\partial V(y(s))}{\partial
y_{1}}d\hat{B}_{1}(s)\right)\cr\disp &&\hspace{-0.6cm}
+\mathbb{E}\left(\sup_{0\leq u\leq t}\int^{u\wedge\tau_{m}}_{0}\frac{\sigma_{3}}{r^{n-\frac{1}{2}}}\frac{\partial V(y(s))}{\partial
y_{n}}d\hat{B}_{3}(s)\right).
\end{eqnarray}
By Assumption (A1),  $\int^{t\wedge \tau_{m}}_{0}-\frac{\sigma_{1}\alpha_{1}\sqrt{2\beta_{1}}}{\sqrt{r}}\frac{\partial V(y(s))}{\partial y_{1}}d\hat{B}_{1}(s)$ is a martingale on $t\geq 0$,
 and so is for  $\int^{t\wedge \tau_{m}}_{0}\frac{\sigma_{3}}{r^{n-\frac{1}{2}}}\frac{\partial V(y(s))}{\partial y_{n}}d\hat{B}_{3}(s)$.
By the  Burkholder-Davis-Gundy inequality (\cite[Theorem 1.7.3]{mao}),
\begin{eqnarray}\label{fdf21}
\disp && \mathbb{E}\left(\sup_{0\leq u\leq t}\int^{u\wedge\tau_{m}}_{0}-\frac{\sigma_{1}\alpha_{1}\sqrt{2\beta_{1}}}{\sqrt{r}}\frac{\partial V(y(s))}{\partial
y_{1}}d\hat{B}_{1}(s)\right) \cr\disp &&\hspace{-0.2cm}
\leq 4\sqrt{2}\mathbb{E}\left(\int^{t}_{0}|\frac{\sigma_{1}\alpha_{1}\sqrt{2\beta_{1}}}{\sqrt{r}}\frac{\partial V(y(s\wedge\tau_{m}))}{\partial
y_{1}}|^{2}ds\right)^{\frac{1}{2}} \cr\disp &&\hspace{-0.2cm}
\leq\frac{16c^{2}_{1}\sigma^{2}_{1}\alpha^{2}_{1}\beta_{1}}{r}+\frac{1}{\lambda_{1}}\int^{t}_{0}\mathbb{E}V(y(s\wedge\tau_{m}))ds,
\end{eqnarray}
and
\begin{eqnarray}\label{fdf22}
\disp &&\hspace{-0.4cm} \mathbb{E}\left(\sup_{0\leq u\leq t}\int^{u\wedge\tau_{m}}_{0}\frac{\sigma_{3}}{r^{n-\frac{1}{2}}}\frac{\partial V(y(s))}{\partial
y_{n}}d\hat{B}_{3}(s)\right) \cr \disp&& \hspace{-0.4cm}
\leq \frac{4\sqrt{2}\sigma_{3}}{r^{n-\frac{1}{2}}}\mathbb{E}\left(\int^{t}_{0}|\frac{\partial V(y(s\wedge\tau_{m}))}{\partial
y_{n}}|^{2}ds\right)^{\frac{1}{2}}\cr\disp &&\hspace{-0.4cm}\leq \frac{4\sqrt{2}\sigma_{3}c_{1}}{r^{n-\frac{1}{2}}\sqrt{\lambda_{1}}}\mathbb{E}\left(\sup_{0\leq s\leq t}V(y(s\wedge\tau_{m}))\int^{t}_{0}1ds\right)^{\frac{1}{2}}\cr\disp &&\hspace{-0.4cm}
\leq \frac{1}{2}\mathbb{E}(\sup_{0\leq s\leq t}V(y(s\wedge\tau_{m})))+\int^{t}_{0}\frac{16c^{2}_{1}\sigma^{2}_{3}}{r^{2n-1}\lambda_{1}}ds.
\end{eqnarray}
Combining (\ref{fdf20}), (\ref{fdf21}) and  (\ref{fdf22}),  we obtain
\begin{eqnarray}
\disp &&\hspace{-0.6cm} \mathbb{E}\left(\sup_{0\leq s\leq t}V(y(s\wedge \tau_{m}))\right)
\cr&&\hspace{-0.9cm} \leq N_{1}+\int^{t}_{0}(N_{2}+N_{3}\mathbb{E}V(y(s\wedge\tau_{m})))ds
\cr&&\hspace{-0.9cm}= N_{1}+N_{3}\int^{t}_{0}(\frac{N_{2}}{N_{3}}+\mathbb{E}\sup_{0\leq u\leq s}V(y(u\wedge\tau_{m})))ds,
\end{eqnarray}
where we set
\begin{eqnarray}\label{constantsds}
\disp && N_{1}=2\mathbb{E}V(y(0))+\frac{32c^{2}_{1}\sigma^{2}_{1}\alpha^{2}_{1}\beta_{1}}{r}, \cr\disp&&
 N_{2}=\frac{c^{2}_{1}M^{2}}{r}+\frac{c^{2}_{1}\sigma^{2}_{1}\alpha^{2}_{1}\gamma_{1}}{r}+\frac{2c_{2}\sigma^{2}_{1}\alpha^{2}_{1}\beta_{1}}{r}
+\frac{c^{2}_{1}\sigma^{2}_{2}\gamma_{2}}{r^{n}}\cr\disp&&+\frac{c_{2}\sigma^{2}_{3}}{r^{2n-1}}+\frac{32c^{2}_{1}\sigma^{2}_{3}}{r^{2n-1}\lambda_{1}},  \cr\disp&& N_{3}=\frac{2}{\lambda_{1}r}+\frac{1}{\lambda_{1}r^{n}}+\frac{2}{\lambda_{1}}.
\end{eqnarray}
Now, applying  Gronwall's inequality (\cite[Theorem 1.8.1]{mao}) yields
\begin{equation}
 \frac{N_{2}}{N_{3}}+\mathbb{E}\left(\sup_{0\leq s\leq t}V(y(s\wedge\tau_{m}))\right) \leq
\left(N_{1}+\frac{N_{2}}{N_{3}}\right)e^{N_{3}t}.
\end{equation}
Thus,
\begin{equation}\label{315fdv}
\begin{array}{l}
\disp  \mathbb{E}\left(\sup_{0\leq s\leq t\wedge\tau_{m}}\|y(s)\|^{2}\right)\crr \disp
 \leq
\frac{1}{\lambda_{1}}\mathbb{E}\left(\sup_{0\leq s\leq t\wedge\tau_{m}}V(y(s))\right)\leq \frac{1}{\lambda_{1}}\left(N_{1}+\frac{N_{2}}{N_{3}}\right)e^{N_{3}t}.
\end{array}
\end{equation}
This implies that
\begin{eqnarray}
\disp &&\hspace{-0.9cm} m^{2}P\{\tau_{m}\leq t\}\leq \frac{1}{\lambda_{1}}\left(N_{1}+\frac{N_{2}}{N_{3}}\right)e^{N_{3}t}.
\end{eqnarray}
Passing to the limit as  $m \rightarrow\infty$ gives  $P\{\tau\leq t\}=0$ which yields in turn $P\{\tau> t\}=1$.
Since $t\geq 0$ is arbitrary, we then have $\tau=\infty$ almost surely, and   $y(t)$ exists  globally.
Furthermore, passing to  the limit as  $m \rightarrow\infty$ for (\ref{315fdv}) again gives (\ref{fde4})  from Fatou's Lemma.
This completes the proof of the Lemma \ref{lemmasef}.
\begin{center}
APPENDIX B:  Proof of Theorem \ref{theorem}
\end{center}

It follows from Lemma \ref{lemmasef} that the TD (\ref{TD}) admits a unique global solution and
$\int^{t}_{0}\frac{\sigma_{1}\alpha_{1}\sqrt{2\beta_{1}}}{\sqrt{r}}\frac{\partial V(y(s))}{\partial y_{1}}d\hat{B}_{1}(s)$  
is a martingale on $t\geq 0$,   and
 so is for  $\int^{t}_{0}\frac{\sigma_{3}}{r^{n-\frac{1}{2}}}\frac{\partial V(y(s))}{\partial y_{n}}d\hat{B}_{3}(s)$.
Thus, for all $t\geq 0$, it follows that
\begin{eqnarray}\label{3843}
\disp && \mathbb{E}\int^{t}_{0}\frac{\sigma_{1}\alpha_{1}
\sqrt{2\beta_{1}}}{\sqrt{r}}\frac{\partial V(y(s))}{\partial y_{1}}d\hat{B}_{1}(s)=0, \cr\disp &&
\mathbb{E}\int^{t}_{0}\frac{\sigma_{3}}{r^{n-\frac{1}{2}}}\frac{\partial V(y(s))}{\partial
y_{n}}d\hat{B}_{3}(s)=0.
\end{eqnarray}
 Finding the derivative of $\mathbb{E}V(y(t))$ with respect to $t$, it follows
from Assumption (A1), (\ref{11boundedness}), (\ref{198f}), (\ref{3843}), $r\in \mathbb{R}_{0}$ and Young's inequality that
\begin{eqnarray}
\disp && \frac{d\mathbb{E}V(y(t))}{dt}\cr\disp&& =
 \mathbb{E}\left(\sum^{n-1}_{i=1}\frac{\partial V(y(t))}{\partial
y_{i}}y_{i+1}(t)+\frac{\partial V(y(t))}{\partial
y_{n}}f(y(t))\right)
\cr\disp&& +\mathbb{E}\left(\frac{\partial V(y(t))}{\partial
y_{1}}(-\frac{1}{r}\frac{dv(u)}{du}|_{u=\frac{t}{r}}+\frac{\sigma_{1}\alpha_{1}}{r}w_{1}(\frac{t}{r}))\right)
\cr\disp&&  +\mathbb{E}\left(\frac{\sigma^{2}_{1}\alpha^{2}_{1}\beta_{1}}{r}\frac{\partial^{2} V(y(t))}{\partial y^{2}_{1}}
+\frac{\sigma_{2}}{r^{n}}\frac{\partial V(y(t))}{\partial y_{n}}w_{2}(\frac{t}{r})\right)
\cr\disp && + \frac{\sigma^{2}_{3}}{2r^{2n-1}}\mathbb{E}\frac{\partial^{2} V(y(t))}{\partial y^{2}_{n}}
\cr\disp&&\leq -\mathbb{E}W(y(t))+\frac{\lambda_{1}}{2r}\mathbb{E}\|y(t)\|^{2}
+\frac{c^{2}_{1}M^{2}}{2\lambda_{1}r}+\frac{\lambda_{1}}{2r}\mathbb{E}\|y(t)\|^{2}
\cr\disp&&+\frac{c^{2}_{1}\sigma^{2}_{1}\alpha^{2}_{1}}{2\lambda_{1}r}\mathbb{E}|w_{1}(\frac{t}{r})|^{2}
+\frac{c_{2}\sigma^{2}_{1}\alpha^{2}_{1}\beta_{1}}{r}+\frac{\lambda_{1}}{2r^{2n-1}}\mathbb{E}\|y(t)\|^{2}
\cr\disp &&+\frac{c^{2}_{1}\sigma^{2}_{2}}{2\lambda_{1}r}\mathbb{E}|w_{2}(\frac{t }{r})|^{2}+\frac{c_{2}\sigma^{2}_{3}}{2r^{2n-1}}
\cr\disp && \leq -\frac{\lambda_{3}}{\lambda_{2}}\mathbb{E}V(y(t))+\frac{1}{2r}\mathbb{E}V(y(t))+\frac{c^{2}_{1}M^{2}}{2\lambda_{1}r}
\cr\disp&& +\frac{1}{2r}\mathbb{E}V(y(t))+\frac{c^{2}_{1}\sigma^{2}_{1}\alpha^{2}_{1}\gamma_{1}}{2\lambda_{1}r}+\frac{c_{2}\sigma^{2}_{1}\alpha^{2}_{1}\beta_{1}}{r}\cr &&+\frac{1}{2r^{2n-1}}\mathbb{E}V(y(t))+\frac{c^{2}_{1}\sigma^{2}_{2}\gamma_{2}}{2\lambda_{1}r}+\frac{c_{2}\sigma^{2}_{3}}{2r^{2n-1}}
\cr\disp&& \leq
-\frac{(1-\theta)\lambda_{3}}{\lambda_{2}}\mathbb{E}V(y(t))+\frac{\Gamma_{1}}{r},
\end{eqnarray}
where
\begin{eqnarray}\label{6153e}
&&\Gamma_{1}:=\frac{c^{2}_{1}M^{2}}{2\lambda_{1}}+\frac{c^{2}_{1}\sigma^{2}_{1}\alpha^{2}_{1}\gamma_{1}}{2\lambda_{1}}+c_{2}\sigma^{2}_{1}\alpha^{2}_{1}\beta_{1}
\cr&&+\frac{c^{2}_{1}\sigma^{2}_{2}\gamma_{2}}{2\lambda_{1}}
+\frac{c_{2}\sigma^{2}_{3}}{2},
\end{eqnarray}
and $\theta\in (0,1)$ is given in (\ref{37df}). This,  together with Assumption (A1), yields that
\begin{eqnarray}\label{515dfd}
&&\mathbb{E}V(y(t))\cr&&\leq e^{-\frac{(1-\theta)\lambda_{3}}{\lambda_{2}}t}\mathbb{E}V(y(0))+\frac{\Gamma_{1}}{r}\int^{t}_{0}e^{-\frac{(1-\theta)\lambda_{3}}{\lambda_{2}}(t-s)}ds\cr&&
\leq \lambda_{2}e^{-\frac{(1-\theta)\lambda_{3}}{\lambda_{2}}t}\mathbb{E}\|y(0)\|^{2}+\frac{\lambda_{2}\Gamma_{1}}{\lambda_{3}(1-\theta)r}.
\end{eqnarray}
Since
\begin{eqnarray}
&&\hspace{-1.2cm}\mathbb{E}\|y(0)\|^{2}\cr
&&\hspace{-1.2cm}=\mathbb{E}|x_{1}(0)-v(0)-\sigma_{1}w_{1}(0)|^{2}+\sum^{n-1}_{i=1}\frac{1}{r^{2i}}|x_{i+1}(0)|^{2},
\end{eqnarray}
it can be concluded that for any $T>0$,
 there exists a  positive constant
\begin{eqnarray}\label{618fd}
&&\Gamma_{2}:= \sup_{r\in R_{0}}(e^{-\frac{(1-\theta)\lambda_{3}}{\lambda_{2}}rT}r)\cdot[\mathbb{E}|x_{1}(0)-v(0)-\sigma_{1}w_{1}(0)|^{2}
\cr&&\hspace{0.5cm}+\sum^{n-1}_{i=1}|x_{i+1}(0)|^{2}]
\end{eqnarray}
 which is independent of $r$. This is because $g(r):=e^{-\frac{(1-\theta)\lambda_{3}}{\lambda_{2}}rT}r$ is  continuous with respect to
$r$. Since $\lim_{r\to \infty}g(r)=0$,  $g(r)$ is  bounded  on the domain $R_{0}$,
i.e., there is a positive constant $N$ independent of $r$ such that $N=\sup_{r\in R_{0}}(e^{-\frac{(1-\theta)\lambda_{3}}{\lambda_{2}}rT}r)$.
Hence, $\Gamma_2$ is independent of $r$ and so $e^{-\frac{(1-\theta)\lambda_{3}}{\lambda_{2}}rT}\mathbb{E}\|y(0)\|^{2}\leq \frac{\Gamma_{2}}{r}$.
Therefore, for all $t\in [T,\infty)$,
\begin{eqnarray}\label{519df}
\hspace{-1cm}&&\mathbb{E}|x_{1}(t)-v^{*}(t)|^{2}\cr\disp\hspace{-1cm} &&
=\mathbb{E}|y_{1}(rt)|^{2}\leq \mathbb{E}\|y(rt)\|^{2}\leq \frac{1}{\lambda_{1}}\mathbb{E}V(y(rt))
\cr\disp\hspace{-1cm} && \leq \frac{\lambda_{2}}{\lambda_{1}}e^{-\frac{(1-\theta)\lambda_{3}}{\lambda_{2}}rT}\mathbb{E}\|y(0)\|^{2}
+\frac{\lambda_{2}\Gamma_{1}}{\lambda_{3}(1-\theta)\lambda_{1}r}
 \leq \frac{\Gamma}{r},
\end{eqnarray}
where
\begin{eqnarray}\label{734fd}
\Gamma:=\frac{\lambda_{2}\Gamma_{2}}{\lambda_{1}}+\frac{\lambda_{2}\Gamma_{1}}{\lambda_{3}(1-\theta)\lambda_{1}}.
\end{eqnarray}
is a positive constant independent of $r$.
Using the inequality $(a+b)^{2}\leq (1+\frac{1}{\mu})a^{2}+(1+\mu) b^{2}$ for any $\mu>0$ and $a,b\in\mathbb{R}$,  it is
obtained by (\ref{11boundedness}) and (\ref{519df}) that
\begin{eqnarray}\label{330d}
&&\mathbb{E}|x_{1}(t)-v(t)|^{2}\cr\disp&&
\leq (1+\frac{1}{\mu})\mathbb{E}|x_{1}(t)-v^{*}(t)|^{2}+(1+\mu)\sigma^{2}_{1}\mathbb{E}|w_{1}(t)|^{2}
\cr\disp&& \leq \frac{(1+\frac{1}{\mu})\Gamma}{r}+(1+\mu)\sigma^{2}_{1}\gamma_{1}
\end{eqnarray}
uniformly in $t\in [T,\infty)$.
Since $\mu>0$ is arbitrary, it follows from (\ref{330d}) that
\begin{eqnarray}\label{superlimit}
\limsup_{r\to \infty} \mathbb{E}|x_{1}(t)-v(t)|^{2}\leq \sigma^{2}_{1}\gamma_{1}
\end{eqnarray}
uniformly in  $t\in [T,\infty)$. In addition, when $\sigma_{1}=0$, if follows from (\ref{superlimit}) that
\begin{eqnarray}
\lim_{r\to \infty} \mathbb{E}|x_{1}(t)-v(t)|^{2}=0
\end{eqnarray}
uniformly in  $t\in [T,\infty)$. Thus, for any $\varepsilon:=\frac{1}{m^{4}}>0,m\in \mathbb{N}^{+}$, there exists an $m$-dependent constant $r^{*}=r^{*}(m)$
such that
\begin{eqnarray}\label{524dfd}
\mathbb{E}|x_{1}(t)-v(t)|^{2}<\frac{1}{m^{4}}
\end{eqnarray}
uniformly in  $t\in [T,\infty)$ for all $r\geq r^{*}$. By Chebyshev's inequality (\cite[p.5]{mao}), it  has
\begin{eqnarray}
P\left\{\omega:|x_{1}(t)-v(t)|> \frac{1}{m}\right\}\leq\frac{1}{m^{2}}
\end{eqnarray}
uniformly in  $t\in [T,\infty)$ for all $r\geq r^{*}$. By Borel-Cantelli's lemma (\cite[p.7]{mao}), 
it can be also obtained that for almost all $\omega \in \Omega$,
there exists an $m_{0}=m_{0}(\omega)$ such that
\begin{eqnarray}
|x_{1}(t)-v(t)|\leq \frac{1}{m}
\end{eqnarray}
uniformly in  $t\in [T,\infty)$ whenever $m\geq m_{0},r\geq r^{*}$. Therefore, for almost all $\omega \in \Omega$,
\begin{eqnarray}
\disp\limsup_{r\to\infty} |x_{1}(t)-v(t)|\leq \frac{1}{m}
\end{eqnarray}
whenever $m\geq m_{0}$. Setting $m\rightarrow\infty$ gives
\begin{eqnarray}\label{5328}
\disp\lim_{r\to \infty} |x_{1}(t)-v(t)|=0, \; \mbox{almost surely}
\end{eqnarray}
uniformly in  $t\in [T,\infty)$ when $\sigma_{1}=0$. This completes the proof of the Theorem \ref{theorem}.

\begin{center}
APPENDIX C: Proof of Theorem \ref{weakconvergence}
\end{center}
%
By (\ref{38fd}) and performing the integration by parts, it can be easily obtained that for each $i=2,3,\cdots,n$,
\begin{equation}\label{fdfdf}
x_{i}(\varphi)=(-1)^{(i-1)}\int^{a}_{0}x_{1}(t)\varphi^{(i-1)}(t)dt,\; \forall \varphi\in C^{\infty}_{0}(0,a).
\end{equation}

From Theorem \ref{theorem}, (\ref{fdfdf}) and the definition of the  generalized derivative
in (\ref{definitiondervc}), for each $i=2,3,\cdots,n$ and
any $0<\xi<a$, it follows that
\begin{eqnarray}
&& \hspace{-0.5cm}\mathbb{E}|x_{i}(\varphi)-v^{(i-1)}(\varphi)|^{2} \cr\disp&& \hspace{-0.5cm}
=\mathbb{E}|\int^{a}_{0}x_{i}(t)\varphi(t)dt-(-1)^{(i-1)}\int^{a}_{0}v(t)\varphi^{(i-1)}(t)dt|^{2}\cr\disp&& \hspace{-0.5cm}
=\mathbb{E}|\int^{a}_{0}(x_{1}(t)-v(t))\varphi^{(i-1)}(t)dt|^{2}\cr\disp&& \hspace{-0.5cm}\leq
a\int^{a}_{0}\mathbb{E}|x_{1}(t)-v(t)|^{2}dt\sup_{t\in(0,a)}|\varphi^{(i-1)}(t)|^{2} \cr\disp &&\hspace{-0.5cm}\leq
a\int^{\xi}_{0}\mathbb{E}|x_{1}(t)-v(t)|^{2}dt\sup_{t\in(0,a)}|\varphi^{(i-1)}(t)|^{2} \cr\disp && \hspace{-0.5cm}
+a\int^{a}_{\xi}\mathbb{E}|x_{1}(t)-v(t)|^{2}dt\sup_{t\in(0,a)}|\varphi^{(i-1)}(t)|^{2} \cr\disp && \hspace{-0.5cm}
\leq \xi a\max_{0\leq t\leq \xi}\mathbb{E}|x_{1}(t)-v(t)|^{2}\sup_{t\in(0,a)}|\varphi^{(i-1)}(t)|^{2}
\cr\disp && \hspace{-0.5cm} + a(a-\xi)\sup_{t\in(0,a)}|\varphi^{(i-1)}(t)|^{2}({(1+\frac{1}{\mu})
\frac{\Gamma}{r}}\cr\disp &&\hspace{-0.5cm} +(1+\mu)\sigma^{2}_{1}\gamma_{1}).
\end{eqnarray}
 Since $\mu>0$ is arbitrary, passing to the limit as $r\rightarrow\infty$ yields
\begin{eqnarray}
&& \limsup_{r\to \infty} \mathbb{E}|x_{i}(\varphi)-v^{(i-1)}(\varphi)|^{2} \cr\disp&&
\leq \xi a\max_{0\leq t\leq \xi}\mathbb{E}|x_{1}(t)-v(t)|^{2}\sup_{t\in(0,a)}|\varphi^{(i-1)}(t)|^{2}
\cr\disp&& + a(a-\xi)\sup_{t\in(0,a)}|\varphi^{(i-1)}(t)|^{2}\sigma^{2}_{1}\gamma_{1}.
\end{eqnarray}
Setting  $\xi\rightarrow 0$, we then have
\begin{eqnarray}\label{529d}
&& \limsup_{r\to \infty} \mathbb{E}|x_{i}(\varphi)-v^{(i-1)}(\varphi)|^{2} \cr\disp&&
\leq  a^{2}\sup_{t\in(0,a)}|\varphi^{(i-1)}(t)|^{2}\sigma^{2}_{1}\gamma_{1}.
\end{eqnarray}
 When $\sigma_{1}=0$, it follows from  (\ref{529d}) that
\begin{eqnarray}
 \lim_{r\to \infty} \mathbb{E}|x_{i}(\varphi)-v^{(i-1)}(\varphi)|^{2}=0.
\end{eqnarray}
Similar to  (\ref{524dfd})-(\ref{5328}), it can be also obtained that
\begin{eqnarray}
 \lim_{r\to \infty} |x_{i}(\varphi)-v^{(i-1)}(\varphi)|=0,\;  \mbox{almost surely}.
\end{eqnarray}
This completes the proof of the Theorem \ref{weakconvergence}.

%
%

\ifCLASSOPTIONcaptionsoff
  \newpage
\fi




\begin{thebibliography}{1}
 
\bibitem{Han2009}  Han, J.Q. (2009). From PID to active disturbance rejection control. {\it IEEE Transactions on Industrial Electronics}, 56(3), 900-906.

\bibitem{han1994} Han, J.Q., \& Wang, W. (1994).  Nonlinear tracking-differentiator. {\it Journal of Systems Science and Mathematical Science}, 14, 177-183 (in Chinese).

\bibitem{xue2010}Xue, W.C., Huang, Y., \& Yang, X.X. (2010, July). What kinds of system can be used as tracking-differentiator. {\it In Proceedings of the 29th Chinese Control Conference (pp. 6113-6120). IEEE.}

\bibitem{su2005} Su, Y.X., Zheng, C.H., Sun, D., \& Duan, B.Y. (2005). A simple nonlinear velocity estimator for high-performance motion control.
 {\it IEEE Transactions on Industrial Electronics}, 52(4), 1161-1169.


\bibitem{tian2013} Tian, D.P., Shen, H.H., \& Dai, M. (2013). Improving the rapidity of nonlinear tracking differentiator via feedforward. {\it IEEE Transactions on Industrial Electronics}, 61(7), 3736-3743.

\bibitem{shen2017}Shen, J., Xin, B., Cui, H.Q., \& Gao, W.X. (2017). Control of single-axis rotation INS by tracking differentiator based fuzzy PID. {\it IEEE Transactions on Aerospace and Electronic Systems}, 53(6), 2976-2986.


\bibitem{zhang2018} Zhang, H.H., Xie, Y.D., Xiao, G.X., Zhai, C., \& Long, Z.Q. (2018). A simple discrete-time tracking differentiator and its application to speed and position detection system for a maglev train. {\it IEEE Transactions on Control Systems Technology}, 27(4), 1728-1734.

\bibitem{Guo2002}Guo, B.Z., Han, J.Q., \& Xi, F.B. (2002). Linear tracking-differentiator and application to online estimation of the frequency of a sinusoidal signal with random noise perturbation. {\it International Journal of Systems Science}, 33(5), 351-358.

\bibitem{Guo2011} Guo, B.Z., \& Zhao, Z.L. (2011). On convergence of tracking differentiator. {\it International Journal of Control}, 84(4), 693-701.

\bibitem{Guo2012}Guo, B.Z., \& Zhao, Z.L. (2012). Weak convergence of nonlinear high-gain tracking differentiator. {\it IEEE Transactions on Automatic Control}, 58(4), 1074-1080.

\bibitem{Guo2016} Guo, B.Z., \&  Zhao, Z.L. (2016). {\it Active Disturbance Rejection
Control for Nonlinear Systems: An Introduction.} John Wiley \&
Sons, New York.


\bibitem{krtics}Deng, H., \& Krsti\'{c}, M. (1997). Stochastic nonlinear stabilization-I: a backstepping design. {\it Systems \& Control Letters}, 32(3), 143-150.

\bibitem{Pan2}Pan, Z., \& Basar, T. (1999). Backstepping controller design for nonlinear stochastic systems under a risk-sensitive cost criterion. {\it SIAM Journal on Control and Optimization}, 37(3), 957-995.

\bibitem{bask3}Deng, H., \& Krsti\'{c}, M. (1999). Output-feedback stochastic nonlinear stabilization. {\it IEEE Transactions on Automatic Control}, 44(2), 328-333.

\bibitem{ZhaoDeng2022tac} Zhao, X.Y., \& Deng, F.Q. (2022).  Time-varying Halanay inequalities with application
to stability and control of delayed stochastic systems.{\it IEEE Trans. Automat. Control},  67(3), 1226-1240.

 


\bibitem{duan2015} Duan, J.Q. (2015). {\it An introduction to stochastic dynamics (Vol. 51).} Cambridge University Press.

\bibitem{colorednoise} Klosek-Dygas, M.M., Matkowsky, B.J., \& Schuss, Z. (1988). Colored noise in dynamical systems. {\it SIAM Journal on Applied Mathematics}, 48(2), 425-441.

\bibitem{colourednoise2} H\"{a}nggi, P., \& Jung, P. (1995). Colored noise in dynamical systems. {\it Advances in chemical physics}, 89, 239-326.

\bibitem{mao} Mao, X.R. (2007). {\it Stochastic Differential Equations and Applications}. Horwood Publishing Limited, Chichester.

\bibitem{Ti} Texas Instruments, Technical Reference Manual. (2013). TMS320F28069M,
 TMS320F28068M InstaSPIN$^{\mbox{{\small TM}}}$-MOTION Software,
  Literature Number: SPRUHJ0A.  April 2013, Revised November 2013.

\end{thebibliography}
%

%






\end{document}